\documentclass{article}%
\usepackage{amsmath}
\usepackage{amsfonts}
\usepackage{amssymb}
\usepackage{graphicx}%
\providecommand{\U}[1]{\protect \rule{.1in}{.1in}}
\newtheorem{theorem}{Theorem}

\newtheorem{corollary}[theorem]{Corollary}

\newtheorem{definition}[theorem]{Definition}

\newtheorem{lemma}[theorem]{Lemma}

\newtheorem{proposition}[theorem]{Proposition}
\newtheorem{remark}[theorem]{Remark}

\newenvironment{proof}[1][Proof]{\noindent \textbf{#1.} }{\  $\Box$}
\begin{document}

\title{On the integral representation of $g$-expectations with terminal constraints}
\author{Xiaojuan Li \thanks{School of Information Engineering, Shandong Youth University Of
Political Science, lxj110055\symbol{64}126.com. Research supported by the Natural Science Foundation of Shandong Province(No. ZR2014AP005)}
}
\date{}

\maketitle

\noindent \textbf{Abstract. }In this paper, we study the integral
representation of $g$-expectations with two kinds of terminal constraints, and
obtain the corresponding necessary and sufficient conditions.

\  \  \


\noindent \textbf{MSC-classification}: 60H10, 60H30

\section{Introduction}

Pardoux and Peng \cite{PP} showed that the following type of nonlinear
backward stochastic differential equation (BSDE for short)%
\[
Y_{t}=\xi+\int_{t}^{T}g(s,Y_{s},Z_{s})ds-\int_{t}^{T}Z_{s}dW_{s}%
\]
has a unique solution $(Y,Z)$ under some conditions on $g$, where $\xi$ is
called terminal value and $g$ is called the generator. Based on the solution
of BSDEs, Peng \cite{P2} introduced the notion of $g$-expectations
$\mathcal{E}_{g}[\cdot]:L^{2}(\mathcal{F}_{T})\rightarrow \mathbb{R}$, which is
the first kind of dynamically consistent nonlinear expectations. Moreover,
Coquet et al. \cite{CHMP} proved that any dynamically consistent nonlinear
expectation on $L^{2}(\mathcal{F}_{T})$ under certain conditions is $g$-expectation.

One problem of $g$-expectation is to find the condition of $g$ under which the following integral
representation
\begin{equation}
\mathcal{E}_{g}[\xi]=\int_{-\infty}^{0}(\mathcal{E}_{g}[I_{\{ \xi \geq
t\}}]-1)dt+\int_{0}^{\infty}\mathcal{E}_{g}[I_{\{ \xi \geq t\}}]dt \label{RE}%
\end{equation}
holds. Chen et al. \cite{CCD} proved that the integral representation (\ref{RE})
holds for each $\xi \in L^{2}(\mathcal{F}_{T})$ if and only if $\mathcal{E}%
_{g}[\cdot]$ is a classical linear expectation under the assumptions: $g$ is
continuous in $t$ and $W$ is $1$-dimensional Brownian motion. Without these
assumptions on $g$ and $W$, Hu \cite{Hu1, Hu2} showed that the above result on
integral representation (\ref{RE}) for each $\xi \in L^{2}(\mathcal{F}_{T})$
still holds. For the integral representation (\ref{RE}) with terminal
constraints on $\xi=\Phi(X_{T})$, where $\Phi$ is a monotonic function and $X$
is a solution of stochastic differential equation (SDE for short), Chen et al.
\cite{CS, CKW} obtained a necessary and sufficient condition under the above
assumptions on $g$ and $W$, and gave a sufficient condition for
multi-dimensional Brownian motion.

In this paper, we want to study the integral representation (\ref{RE}) with the following
two kinds of terminal constraints on $\xi=\Phi(X_{T})$: one is for the monotonic
$\Phi$, the other is for the measurable $\Phi$. Specially, we make further research to the structure of $Z$ in the
BSDE and apply it to obtain the corresponding necessary and sufficient conditions without
the above assumptions on $g$ and $W$, which is weaker than the sufficient
condition in \cite{CKW} (see Remark \ref{NeRe6} in Section 3 for detailed
explanation). Furthermore, this method can be extended to solve more general terminal
constraints on $\xi$.

This paper is organized as follows: In Section 2, we recall some basic results
of BSDEs and $g$-expectations. The main result is stated and proved in Section 3.

\section{Preliminaries}

Let $(W_{t})_{t\geq0}=(W_{t}^{1},\ldots,W_{t}^{d})_{t\geq0}$ be a
$d$-dimensional standard Brownian motion defined on a completed probability
space $(\Omega,\mathcal{F},P)$ and $(\mathcal{F}_{t})_{0\leq t\leq T}$ be the
natural filtration generated by this Brownian motion, i.e.,%
\[
\mathcal{F}_{t}:=\sigma \{W_{s}:s\leq t\} \vee \mathcal{N},
\]
where $\mathcal{N}$ is the set of all $P$-null subsets. Fix $T>0$, we denote
by $L^{2}(\mathcal{F}_{t};\mathbb{R}^{m})$, $t\in \lbrack0,T]$, the set of all
$\mathbb{R}^{m}$-valued square integrable $\mathcal{F}_{t}$-measurable random
vectors and $L^{2}(0,T;\mathbb{R}^{m})$ the space of all progressively
measurable, $\mathbb{R}^{m}$-valued processes $(a_{t})_{t\in \lbrack0,T]}$ with
$E[\int_{0}^{T}|a_{t}|^{2}dt]<\infty$.

We consider the following forward-backward stochastic differential equations:%
\begin{equation}
\left \{
\begin{array}
[c]{l}%
dX_{s}^{t,x}=b(s,X_{s}^{t,x})ds+\sigma(s,X_{s}^{t,x})dW_{s},\ s\in \lbrack
t,T],\\
X_{t}^{t,x}=x\in \mathbb{R}^{n},
\end{array}
\right.  \label{SDE1}%
\end{equation}%
\begin{equation}
y_{s}^{t,x}=\Phi(X_{T}^{t,x})+\int_{s}^{T}g(r,y_{r}^{t,x},z_{r}^{t,x}%
)dr-\int_{s}^{T}z_{r}^{t,x}dW_{r}. \label{BSDE1}%
\end{equation}
In this paper, we use the following assumptions:

\begin{description}
\item[(S1)] $b:[0,T]\times \mathbb{R}^{n}\rightarrow \mathbb{R}^{n}$,
$\sigma:[0,T]\times \mathbb{R}^{n}\rightarrow \mathbb{R}^{n\times d}$ are measurable.

\item[(S2)] There exists a constant $K_{1}\geq0$ such that%
\[
|b(t,x)-b(t,x^{\prime})|+|\sigma(t,x)-\sigma(t,x^{\prime})|\leq K_{1}%
|x-x^{\prime}|,\  \forall t\leq T,x,x^{\prime}\in \mathbb{R}^{n}.
\]

\item[(S3)] $\int_{0}^{T}(|b(t,0)|^{2}+|\sigma(t,0)|^{2})dt<\infty$.

\item[(H1)] $g:[0,T]\times \mathbb{R}\times \mathbb{R}^{d}\rightarrow \mathbb{R}$
is measurable.

\item[(H2)] There exists a constant $K_{2}\geq0$ such that%
\[
|g(t,y,z)-g(t,y^{\prime},z^{\prime})|\leq K_{2}(|y-y^{\prime}|+|z-z^{\prime
}|),\  \forall t\leq T,y,y^{\prime}\in \mathbb{R},z,z^{\prime}\in \mathbb{R}%
^{d}.
\]

\item[(H3)] $g(t,y,0)\equiv0$ for each $(t,y)\in \lbrack0,T]\times \mathbb{R}.$

\item[(H3')] $\int_{0}^{T}|g(t,0,0)|^{2}dt<\infty$.

\item[(H4)] $\Phi:\mathbb{R}^{n}\rightarrow \mathbb{R}$ is measurable and
satisfies $\Phi(X_{T}^{t,x})\in L^{2}(\mathcal{F}_{T})$.
\end{description}

\begin{remark}
\label{NeRe1}Obviously, (H3) implies (H3').
\end{remark}

It is well-known that the SDE (\ref{SDE1}) has a unique solution $(X_{s}%
^{t,x})_{s\in \lbrack t,T]}\in L^{2}(t,T;\mathbb{R}^{n})$ under the assumptions
(S1)-(S3). Under the assumptions (H1), (H2), (H3') and (H4), Pardoux and Peng
\cite{PP} showed that the BSDE (\ref{BSDE1}) has a unique solution
$(y_{s}^{t,x},z_{s}^{t,x})_{s\in \lbrack0,T]}\in L^{2}(0,T;\mathbb{R}^{1+d})$.
Moreover, the following result holds.

\begin{theorem}
\label{Th1}(\cite{EPQ, PP1}) Suppose (S1)-(S3), (H1), (H2), (H3') and (H4)
hold. If $b$, $\sigma$, $g$ and $\Phi \in C_{b}^{1,3}$, then

\begin{description}
\item[(i)] $u(t,x):=y_{t}^{t,x}\in C^{1,2}([0,T]\times \mathbb{R}^{n})$ and
solves the following PDE:%
\[
\left \{
\begin{array}
[c]{l}%
\partial_{t}u(t,x)+\mathcal{L}u(t,x)+g(t,u(t,x),\sigma^{T}(t,x)\partial
_{x}u(t,x))=0,\\
u(T,x)=\Phi(x),
\end{array}
\right.
\]
where%
\[
\mathcal{L}u(t,x)=\frac{1}{2}\sum_{i,j=1}^{n}(\sigma \sigma^{T})_{ij}%
(t,x)\partial_{x_{i}x_{j}}^{2}u(t,x)+\sum_{i=1}^{n}b_{i}(t,x)\partial_{x_{i}%
}u(t,x).
\]

\item[(ii)] $z_{s}^{t,x}=\sigma^{T}(s,X_{s}^{t,x})\partial_{x}u(s,X_{s}%
^{t,x})$, $s\in \lbrack t,T]$.
\end{description}
\end{theorem}

\begin{remark}
\label{Re1}For notation simplicity, when $t=0$ and only one $x$, we write
$(X_{t},y_{t},z_{t})_{t\in \lbrack0,T]}$ for the solution of SDE (\ref{SDE1})
and BSDE (\ref{BSDE1}) in the following.
\end{remark}

Using the solution of BSDE, Peng \cite{P2} proposed the following consistent
nonlinear expectations.

\begin{definition}
\label{De1} Suppose $g$ satisfies (H1)-(H3). Let $(y_{t,}z_{t})_{t\in
\lbrack0,T]}$ be the solution of BSDE (\ref{BSDE1}) with terminal value
$\xi \in L^{2}(\mathcal{F}_{T})$, i.e.,%
\[
y_{t}=\xi+\int_{t}^{T}g(s,y_{s},z_{s})ds-\int_{t}^{T}z_{s}dW_{s}.
\]
Define%
\[
\mathcal{E}_{g}[\xi|\mathcal{F}_{t}]:=y_{t}\quad \mbox{for each}\ t\in
\lbrack0,T].
\]
$\mathcal{E}_{g}[\xi|\mathcal{F}_{t}]$ is called the conditional
$g$-expectation of $\xi$ with respect to $\mathcal{F}_{t}$. In particular, if
$t=0$, we write $\mathcal{E}_{g}[\xi]$ which is called the $g$-expectation of
$\xi$.
\end{definition}

\begin{remark}
\label{NeRe2}The assumption (H3) is important in the definition of
$g$-expectation. In particular, under the assumptions (H1)-(H3), if $\xi \in
L^{2}(\mathcal{F}_{t_{0}})$ with $t_{0}<T$, then $\mathcal{E}_{g}%
[\xi|\mathcal{F}_{t}]=\xi$ for $t\in \lbrack t_{0},T]$.
\end{remark}

The following standard estimates of BSDEs can be found in \cite{EPQ, P1,
BCHMP}.

\begin{proposition}
\label{Pr1}Suppose $g_{1}$ and $g_{2}$ satisfy (H1), (H2) and (H3'). Let
$(y_{t,}^{i}z_{t}^{i})_{t\in \lbrack0,T]}$ be the solution of BSDE
(\ref{BSDE1}) with the generator $g_{i}$ and terminal value $\xi_{i}\in
L^{2}(\mathcal{F}_{T})$, $i=1,2$. Then there exists a constant $C>0$ depending
on $K_{2}$ and $T$ such that
\[
E[\sup_{0\leq t\leq T}|y_{t}^{1}-y_{t}^{2}|^{2}+\int_{0}^{T}|z_{t}^{1}%
-z_{t}^{2}|^{2}dt]\leq CE[|\xi^{1}-\xi^{2}|^{2}+\int_{0}^{T}|\bar{g}_{t}%
|^{2}dt],
\]
where $\bar{g}_{t}=g_{1}(t,y_{t}^{1},z_{t}^{1})-g_{2}(t,y_{t}^{1},z_{t}^{1})$.
\end{proposition}

Assume $g$ satisfies (H1)-(H3), set%
\[
V_{g}(A):=\mathcal{E}_{g}[I_{A}]\quad \mbox{for each}\ A\in \mathcal{F}_{T}.
\]
It is easy to verify that $V_{g}(\cdot)$ is a capacity, i.e., (i)
$V_{g}(\emptyset)=0$, $V_{g}(\Omega)=1$; (ii) $V_{g}(A)\leq V_{g}(B)$ for each
$A\subset B$. The corresponding Choquet integral (see \cite{Choq}) is defined
as follows:%
\[
\mathcal{C}_{g}[\xi]:=\int_{-\infty}^{0}[V_{g}(\xi \geq t)-1]dt+\int
_{0}^{\infty}V_{g}(\xi \geq t)dt\quad \mbox{for each}\  \xi \in L^{2}%
(\mathcal{F}_{T}).
\]
It is easy to check that $\mathcal{C}_{g}[I_{A}]=\mathcal{E}_{g}[I_{A}]$ for
each $A\in \mathcal{F}_{T}$. Moreover, $|\mathcal{C}_{g}[\xi]|<\infty$ for each
$\xi \in L^{2}(\mathcal{F}_{T})$ (see \cite{HHC}).

\begin{definition}
\label{De2} Two random variables $\xi$ and $\eta$ are called comonotonic if%
\[
\lbrack \xi(\omega)-\xi(\omega^{\prime})][\eta(\omega)-\eta(\omega^{\prime
})]\geq0\quad \mbox{for each}\  \omega,\omega^{\prime}\in \Omega.
\]

\end{definition}

The following properties of Choquet integral can be found in \cite{Choq, De,
Den}.

\begin{description}
\item[(1)] Monotonicity: If $\xi \geq \eta,$ then $\mathcal{C}_{g}[\xi
]\geq \mathcal{C}_{g}[\eta].$

\item[(2)] Positive homogeneity: If $\lambda \geq0,$ then $\mathcal{C}%
_{g}[\lambda \xi]=\lambda \mathcal{C}_{g}[\xi].$

\item[(3)] Translation invariance: If $c\in \mathbb{R},$ then $\mathcal{C}%
_{g}[\xi+c]=\mathcal{C}_{g}[\xi]+c.$

\item[(4)] Comonotonic additivity: If $\xi$ and $\eta$ are comonotonic, then
$\mathcal{C}_{g}[\xi+\eta]=\mathcal{C}_{g}[\xi]+\mathcal{C}_{g}[\eta].$
\end{description}

\section{Main result}

Suppose $n=1$, we define%
\[%
\begin{array}
[c]{l}%
\mathcal{H}:=\{ \xi:\exists b,\sigma \text{ satisfying (S1)-(S3) and }x\text{
such that }\xi=X_{T}^{0,x}\}.\\
\mathcal{H}_{1}:=\{ \Phi(\xi)\in L^{2}(\mathcal{F}_{T}):\Phi \text{ is
monotonic and }\xi \in \mathcal{H}\}.\\
\mathcal{H}_{2}:=\{ \Phi(\xi)\in L^{2}(\mathcal{F}_{T}):\Phi \text{ is
measurable and }\xi \in \mathcal{H}\}.
\end{array}
\]

The elements in $\mathcal{H}_{1}$ and $\mathcal{H}_{2}$ can be seen as the
contingent claims of European option. Now we give our main theorem.

\begin{theorem}
\label{Th2}Suppose $g$ satisfies (H1)-(H3). Then

\begin{description}
\item[(i)] $\mathcal{E}_{g}[\cdot]=\mathcal{C}_{g}[\cdot]$ on $\mathcal{H}%
_{1}$ if and only if $g$ is independent of $y$ and is positively homogeneous
in $z$, i.e., $g(t,\lambda z)=\lambda g(t,z)$ for all $\lambda \geq0$;

\item[(ii)] $\mathcal{E}_{g}[\cdot]=\mathcal{C}_{g}[\cdot]$ on $\mathcal{H}%
_{2}$ if and only if $g$ is independent of $y$ and is homogeneous in $z$,
i.e., $g(t,\lambda z)=\lambda g(t,z)$ for all $\lambda \in \mathbb{R}$.
\end{description}
\end{theorem}

\begin{remark}
\label{NeRe6}In \cite{CKW}, Chen et al. showed that $\mathcal{E}_{g}%
[\cdot]=\mathcal{C}_{g}[\cdot]$ on $\mathcal{H}_{1}$ under the assumption that
$g$ is positively additive, i.e., $g(t,z_{1}+z_{1}^{\prime},\ldots,z_{d}%
+z_{d}^{\prime})=g(t,z_{1},\ldots,z_{d})+g(t,z_{1}^{\prime},\ldots
,z_{d}^{\prime})$ for $z_{i}z_{i}^{\prime}\geq0$, $i=1,\ldots,d$. Obviously,
this condition on $g$ is stronger than positive homogeneity. For example,
$g(z)=|z|$ is not positively additive, but is positively homogeneous.
\end{remark}

In order to prove this theorem, we need the following lemmas.

\begin{lemma}
\label{Le1}Suppose $g$ satisfies (H1)-(H3). Then for each given $p\in(1,2)$,
there exists a constant $L>0$ depending on $p$, $K_{2}$ and $T$ such that for
each $\xi$, $\eta \in L^{2}(\mathcal{F}_{T})$,%
\[
|\mathcal{C}_{g}[\xi]-\mathcal{C}_{g}[\eta]|\leq L(1+(E[|\xi|^{2}+|\eta
|^{2}])^{\frac{1}{2p}})(E[|\xi-\eta|^{2}])^{\frac{1}{2p}}.
\]
In particular, for each $\xi \in L^{2}(\mathcal{F}_{T})$, we have
$\mathcal{C}_{g}[(\xi \wedge N)\vee(-N)]\rightarrow \mathcal{C}_{g}[\xi]$ as
$N\rightarrow \infty$.
\end{lemma}

\begin{proof}
For each given $p\in(1,2)$, by Proposition 3.2 in Briand et al. \cite{BDH},
there exists a constant $L_{1}>0$ depending on $p$, $K_{2}$ and $T$ such that
for each $\xi$, $\eta \in L^{2}(\mathcal{F}_{T})$,%
\[
|\mathcal{E}_{g}[\xi]-\mathcal{E}_{g}[\eta]|\leq L_{1}(E[|\xi-\eta
|^{p}])^{\frac{1}{p}}.
\]
Set $\bar{g}(t,y,z)=-g(t,1-y,-z)$, it is easy to check that $1-V_{g}%
(A)=V_{\bar{g}}(A^{c})$. Thus $\mathcal{C}_{g}[\xi]=\mathcal{C}_{g}[\xi
^{+}]-\mathcal{C}_{\bar{g}}[\xi^{-}]$. From this we only need to prove the
result for $\xi \geq0$ and $\eta \geq0$. We have%
\begin{align*}
|\mathcal{C}_{g}[\xi]-\mathcal{C}_{g}[\eta]|  &  \leq \int_{0}^{\infty
}|\mathcal{E}_{g}[I_{\{ \xi \geq t\}}]-\mathcal{E}_{g}[I_{\{ \eta \geq
t\}}]|dt\\
&  \leq L_{1}\int_{0}^{\infty}(E[|I_{\{ \xi \geq t\}}-I_{\{ \eta \geq t\}}%
|^{p}])^{\frac{1}{p}}dt\\
&  =L_{1}\int_{0}^{\infty}(E[|I_{\{ \xi \geq t\}}-I_{\{ \eta \geq t\}}%
|])^{\frac{1}{p}}dt,
\end{align*}%
\begin{align*}
\int_{0}^{1}(E[|I_{\{ \xi \geq t\}}-I_{\{ \eta \geq t\}}|])^{\frac{1}{p}}dt  &
\leq(E[\int_{0}^{1}|I_{\{ \xi \geq t\}}-I_{\{ \eta \geq t\}}|dt])^{\frac{1}{p}%
}\\
&  =(E[\int_{0}^{1}I_{\{ \xi \wedge \eta<t\leq \xi \vee \eta \}}dt])^{\frac{1}{p}}\\
&  \leq(E[|\xi-\eta|])^{\frac{1}{p}}\\
&  \leq(E[|\xi-\eta|^{2}])^{\frac{1}{2p}},
\end{align*}%
\begin{align*}
\int_{1}^{\infty}(E[|I_{\{ \xi \geq t\}}-I_{\{ \eta \geq t\}}|])^{\frac{1}{p}%
}dt  &  \leq(\int_{1}^{\infty}t^{-\frac{q}{p}}dt)^{\frac{1}{q}}(E[\int
_{1}^{\infty}t|I_{\{ \xi \geq t\}}-I_{\{ \eta \geq t\}}|dt])^{\frac{1}{p}}\\
&  =(\frac{p-1}{2-p})^{\frac{p-1}{p}}(E[\int_{1}^{\infty}tI_{\{ \xi \wedge
\eta<t\leq \xi \vee \eta \}}dt])^{\frac{1}{p}}\\
&  \leq(\frac{p-1}{2-p})^{\frac{p-1}{p}}(\frac{1}{2}E[|\xi^{2}-\eta
^{2}|])^{\frac{1}{p}}\\
&  \leq(\frac{p-1}{2-p})^{\frac{p-1}{p}}(\frac{1}{2})^{\frac{1}{2p}}%
(E[|\xi|^{2}+|\eta|^{2}])^{\frac{1}{2p}}(E[|\xi-\eta|^{2}])^{\frac{1}{2p}},
\end{align*}
where $\frac{1}{p}+\frac{1}{q}=1$. Thus we obtain the result.
\end{proof}

\begin{lemma}
\label{Le2}Let $b$, $\sigma$ satisfy (S1)-(S3), $g$ satisfy (H1)-(H3) and
$\Phi \in C_{b}^{3}$. Then there exist $b_{k}$, $\sigma_{k}$, $g_{k}\in
C_{b}^{1,3}$, $k\geq1$, such that
\[
E[\sup_{t\in \lbrack0,T]}|X_{t}^{k}-X_{t}|^{2}+\int_{0}^{T}(|\sigma_{k}%
(t,X_{t}^{k})-\sigma(t,X_{t})|^{2}+|z_{t}^{k}-z_{t}|^{2})dt]\rightarrow0,
\]
where $(X_{t},y_{t},z_{t})_{t\in \lbrack0,T]}$ is the solution corresponding to
$b$, $\sigma$, $g$ and $(X_{t}^{k},y_{t}^{k},z_{t}^{k})_{t\in \lbrack0,T]}$ is
the solution corresponding to $b_{k}$, $\sigma_{k}$, $g_{k}$.
\end{lemma}

\begin{proof}
By the standard estimates of SDEs and Proposition \ref{Pr1}, we only need to
prove the result for bounded $b$, $\sigma$ and $g$. For any function $h(u)$,
$u\in \mathbb{R}^{m}$, we will denote, for each $\varepsilon>0$,%
\[
h_{\varepsilon}(u)=\int_{\mathbb{R}^{m}}h(u-v)\varepsilon^{-m}\varphi(\frac
{v}{\varepsilon})dv,
\]
where $\varphi$ is the mollifier in $\mathbb{R}^{m}$ defined by $\varphi
(u)=\exp(-\frac{1}{1-|u|^{2}})I_{\{|u|<1\}}$. By this definition, it is easy
to check that $b_{\varepsilon}$, $\sigma_{\varepsilon}$ and $g_{\varepsilon}$
satisfy (S2) and (H2) with the same Lipschitz constant. Also, we have
$b_{\varepsilon}$, $\sigma_{\varepsilon}$, $g_{\varepsilon}\in C_{b}^{1,3}$
and $(b_{\varepsilon},\sigma_{\varepsilon},g_{\varepsilon})\rightarrow
(b,\sigma,g)$ a.e. in $t$ for each fixed $(x,y,z)\in \mathbb{R}^{2+d}$. Thus by
the diagonal method, we can choose a sequence $b_{k}$, $\sigma_{k}$, $g_{k}\in
C_{b}^{1,3}$ such that $(b_{k},\sigma_{k},g_{k})\rightarrow(b,\sigma,g)$ for
every $(x,y,z)\in \mathbb{Q}^{2+d}$ a.e. in $t$. By the Lipschitz condition, we
get $(b_{k},\sigma_{k},g_{k})\rightarrow(b,\sigma,g)$ for every $(x,y,z)\in
\mathbb{R}^{2+d}$ a.e. in $t$. By the estimates of SDEs, we obtain%
\[
E[\sup_{t\in \lbrack0,T]}|X_{t}^{k}-X_{t}|^{2}]\leq L_{2}E[\int_{0}^{T}%
(|b_{k}(t,X_{t})-b(t,X_{t})|^{2}+|\sigma_{k}(t,X_{t})-\sigma(t,X_{t}%
)|^{2})dt],
\]
where the constant $L_{2}$ depending on $K_{1}$ and $T$. By the bounded
dominated convergence theorem, we can get $E[\sup_{t\in \lbrack0,T]}|X_{t}%
^{k}-X_{t}|^{2}]\rightarrow0$. From this, it is easy to deduce that
$E[\int_{0}^{T}|\sigma_{k}(t,X_{t}^{k})-\sigma(t,X_{t})|^{2}dt]\rightarrow0$.
By Proposition \ref{Pr1}, we can easily obtain $E[\int_{0}^{T}|z_{t}^{k}%
-z_{t}|^{2}dt]\rightarrow0$.
\end{proof}

We now prove the main theorem.

\textbf{Proof of Theorem \ref{Th2}. }We first prove that the condition on $g$
is necessary, and then it is sufficient.

(i) Necessity. We first prove the result for the case $d=1$. For this we
choose $b(s,x)=0$, $\sigma(s,x)=zI_{[t,t+\varepsilon]}(s)$ and $\Phi(x)=x$,
where $z\in \mathbb{R}$, $t<T$ and $\varepsilon>0$ are given. Then%
\[
\mathcal{H}_{1}\supset \{y+z(W_{t+\varepsilon}-W_{t}):\forall y,z\in
\mathbb{R,}t<T,\varepsilon>0\}.
\]
Since $\mathcal{E}_{g}[\cdot]=\mathcal{C}_{g}[\cdot]$ on $\mathcal{H}_{1}$ and
$g$ is deterministic, by the properties of $\mathcal{C}_{g}[\cdot]$ we can get%
\[
\mathcal{E}_{g}[y+z(W_{t+\varepsilon}-W_{t})|\mathcal{F}_{t}]=\mathcal{E}%
_{g}[y+z(W_{t+\varepsilon}-W_{t})]=\mathcal{E}_{g}[z(W_{t+\varepsilon}%
-W_{t})|\mathcal{F}_{t}]+y,
\]%
\[
\mathcal{E}_{g}[\lambda z(W_{t+\varepsilon}-W_{t})|\mathcal{F}_{t}%
]=\lambda \mathcal{E}_{g}[z(W_{t+\varepsilon}-W_{t})|\mathcal{F}_{t}]\text{ for
}\lambda \geq0.
\]
By Lemma 2.1 in Jiang \cite{J}, we can obtain that $g$ is independent of $y$
and $g(t,\lambda z)=\lambda g(t,z)$ for all $\lambda \geq0$. For the case
$d>1$. For each given $a\in \mathbb{R}^{d}$ with $|a|=1$, we define $W^{a}$ by
$W_{t}^{a}=a\cdot W_{t}$ and $g^{a}:[0,T]\times \mathbb{R}\times \mathbb{R}%
\rightarrow \mathbb{R}$ by $g^{a}(t,y,z)=g(t,y,az)$. It is easy to check that
$\mathcal{E}_{g}[\xi]=\mathcal{E}_{g^{a}}[\xi]$ and $\mathcal{C}_{g}%
[\xi]=\mathcal{C}_{g^{a}}[\xi]$ for $\xi \in L^{2}(\mathcal{F}_{T}^{a})$, where
$\mathcal{F}_{T}^{a}:=\sigma \{W_{t}^{a}:t\leq T\} \vee \mathcal{N}$. Thus by
applying the method of $d=1$, we can obtain $g^{a}$ is independent of $y$ and
is positively homogeneous in $z$ for each given $a\in \mathbb{R}^{d}$ with
$|a|=1$, which implies the necessary condition on $g$.

Sufficiency. By Proposition \ref{Pr1} and Lemma \ref{Le1}, we only need to
prove the result for bounded and monotonic $\Phi$. The proof is divided into
two steps.

Step 1. Let $(X_{t})_{t\in \lbrack0,T]}$ be the solution of SDE (\ref{SDE1})
corresponding to $b$ and $\sigma$ satisfying (S1)-(S3) and let $\phi_{i}\in
C_{b}^{3}(\mathbb{R})$, $i=1,\ldots,N$, be non decreasing functions. We assert
that%
\begin{equation}
\mathcal{E}_{g}[\sum_{i=1}^{N}\phi_{i}(X_{T})]=\sum_{i=1}^{N}\mathcal{E}%
_{g}[\phi_{i}(X_{T})]. \label{EQ1}%
\end{equation}
Let $(y_{t}^{i},z_{t}^{i})_{t\in \lbrack0,T]}$, $i=1,\ldots,N$, be the solution
of the following BSDEs:%
\begin{equation}
y_{t}^{i}=\phi_{i}(X_{T})+\int_{t}^{T}g(s,z_{s}^{i})ds-\int_{t}^{T}z_{s}%
^{i}dW_{s}. \label{EQ2}%
\end{equation}
By Lemma \ref{Le2}, we can choose $b_{k}$, $\sigma_{k}$, $g_{k}\in C_{b}%
^{1,3}$, $k\geq1$, such that
\[
E[\int_{0}^{T}(|\sigma_{k}(t,X_{t}^{k})-\sigma(t,X_{t})|^{2}+|z_{t}%
^{i,k}-z_{t}^{i}|^{2})dt]\rightarrow0,\text{ }i=1,\ldots,N,
\]
where $(X_{t}^{k},y_{t}^{i,k},z_{t}^{i,k})_{t\in \lbrack0,T]}$ is the solution
corresponding $b_{k}$, $\sigma_{k}$, $g_{k}$ and terminal value $\phi
_{i}(X_{T}^{k})$. From this we can get%
\begin{equation}
z_{t}^{i,k}\rightarrow z_{t}^{i},\text{ }\sigma_{k}(t,X_{t}^{k})\rightarrow
\sigma(t,X_{t})\text{ \ }dP\times dt\text{-a.s..} \label{EQ3}%
\end{equation}
On the other hand, it follows from Theorem \ref{Th1} that%
\begin{equation}
z_{t}^{i,k}=\sigma_{k}^{T}(t,X_{t}^{k})\partial_{x}u^{i,k}(t,X_{t}^{k}),
\label{EQ4}%
\end{equation}
where $u^{i,k}(t,x):=y_{t}^{i,k;t,x}$. By comparison theorem of SDE and BSDE,
it is easy to verify that $u^{i,k}(t,x)$ is non decreasing in $x$, which
implies $\partial_{x}u^{i,k}(t,X_{t}^{k})\geq0$. Thus by combining equation
(\ref{EQ3}) and (\ref{EQ4}), we obtain that there exist progressive processes
$D_{t}^{i}\geq0$, $i=1,\ldots,N$, such that%
\[
z_{t}^{i}=\sigma^{T}(t,X_{t})D_{t}^{i}.
\]
Note that $g$ is positively homogeneous in $z$, then we get%
\begin{align}
\sum_{i=1}^{N}g(t,z_{t}^{i})  &  =\sum_{i=1}^{N}g(t,\sigma^{T}(t,X_{t}%
)D_{t}^{i})=g(t,\sigma^{T}(t,X_{t}))\sum_{i=1}^{N}D_{t}^{i}\nonumber \\
&  =g(t,\sigma^{T}(t,X_{t})\sum_{i=1}^{N}D_{t}^{i})=g(t,\sum_{i=1}^{N}%
z_{t}^{i}). \label{EQ5}%
\end{align}
Set
\[
Y_{t}=\sum_{i=1}^{N}y_{t}^{i},\text{ }Z_{t}=\sum_{i=1}^{N}z_{t}^{i},
\]
then by combining equation (\ref{EQ2}) and (\ref{EQ5}), we can get%
\[
Y_{t}=\sum_{i=1}^{N}\phi_{i}(X_{T})+\int_{t}^{T}g(s,Z_{s})ds-\int_{t}^{T}%
Z_{s}dW_{s}.
\]
By the definition of $g$-expectation, we obtain equation (\ref{EQ1}).

Step 2. Let $(X_{t})_{t\in \lbrack0,T]}$ be as in Step 1 and let $\Phi$ be a
bounded and monotonic function. Note that for each $\xi \in L^{2}%
(\mathcal{F}_{T})$ and $c\in \mathbb{R}$,
\[
\mathcal{E}_{g}[\xi+c]=\mathcal{E}_{g}[\xi]+c,\text{ }\mathcal{C}_{g}%
[\xi+c]=\mathcal{C}_{g}[\xi]+c,
\]
then we only need to prove the result for $\Phi \geq0$. Since the analysis of
non increasing $\Phi$ is the same as in non decreasing $\Phi$, we only prove
the case for non decreasing $\Phi$ with $0\leq \Phi<M$, where $M>0$ is a
constant. For each given $N>0$, we set%
\[
\Phi_{N}(x)=\sum_{i=1}^{N}\frac{(i-1)M}{N}I_{\{ \frac{(i-1)M}{N}\leq \Phi
<\frac{iM}{N}\}}=\sum_{i=1}^{N}\frac{M}{N}I_{\{ \Phi \geq \frac{iM}{N}\}}.
\]
It is easy to check that $E[|\Phi_{N}(X_{T})-\Phi(X_{T})|^{2}]\leq(\frac{M}%
{N})^{2}\rightarrow0$ as $N\rightarrow \infty$. Thus by Proposition \ref{Pr1}
and Lemma \ref{Le1}, we get
\begin{equation}
\mathcal{E}_{g}[\Phi_{N}(X_{T})]\rightarrow \mathcal{E}_{g}[\Phi(X_{T})],\text{
}\mathcal{C}_{g}[\Phi_{N}(X_{T})]\rightarrow \mathcal{C}_{g}[\Phi(X_{T})]\text{
as }N\rightarrow \infty. \label{EQ6}%
\end{equation}
For each fixed $N>0$, noting that $\Phi$ is non decreasing, then $\{ \Phi
\geq \frac{iM}{N}\}$ is $[a_{i},\infty)$ or $(a_{i},\infty)$, where $a_{i}$ is
a constant. For each $\varepsilon>0$, we define%
\[
\psi_{i,\varepsilon}^{1}(x)=\int_{\mathbb{R}}I_{[a_{i}-\varepsilon,\infty
)}(x-v)\frac{1}{\varepsilon}\varphi(\frac{v}{\varepsilon})dv,\psi
_{i,\varepsilon}^{2}(x)=\int_{\mathbb{R}}I_{(a_{i}+\varepsilon,\infty
)}(x-v)\frac{1}{\varepsilon}\varphi(\frac{v}{\varepsilon})dv,
\]
where $\varphi(v)=\exp(-\frac{1}{1-|v|^{2}})I_{\{|v|<1\}}$. It is easy to
check that $\psi_{i,\varepsilon}^{1}$, $\psi_{i,\varepsilon}^{2}\in C_{b}%
^{3}(\mathbb{R})$ are non decreasing and satisfy $\psi_{i,\varepsilon}%
^{1}\downarrow I_{[a_{i},\infty)}$, $\psi_{i,\varepsilon}^{2}\uparrow
I_{(a_{i},\infty)}$ as $\varepsilon \downarrow0$. Thus we can choose non
decreasing $\phi_{i}^{k}\in C_{b}^{3}(\mathbb{R})$, $k\geq1$, such that
$E[|\phi_{i}^{k}(X_{T})-I_{\{ \Phi \geq \frac{iM}{N}\}}(X_{T})|^{2}%
]\rightarrow0$ as $k\rightarrow \infty$, which implies%
\[
E[|\Phi_{N}(X_{T})-\frac{M}{N}\sum_{i=1}^{N}\phi_{i}^{k}(X_{T})|^{2}%
]\rightarrow0\text{ \ as }k\rightarrow \infty.
\]
By Step 1, Proposition \ref{Pr1} and properties of Choquet integral, we can
obtain%
\begin{align*}
\mathcal{E}_{g}[\Phi_{N}(X_{T})]  &  =\lim_{k\rightarrow \infty}\mathcal{E}%
_{g}[\frac{M}{N}\sum_{i=1}^{N}\phi_{i}^{k}(X_{T})]=\lim_{k\rightarrow \infty
}\frac{M}{N}\mathcal{E}_{g}[\sum_{i=1}^{N}\phi_{i}^{k}(X_{T})]\\
&  =\frac{M}{N}\sum_{i=1}^{N}\lim_{k\rightarrow \infty}\mathcal{E}_{g}[\phi
_{i}^{k}(X_{T})]=\frac{M}{N}\sum_{i=1}^{N}\mathcal{E}_{g}[I_{\{ \Phi \geq
\frac{iM}{N}\}}(X_{T})]\\
&  =\frac{M}{N}\sum_{i=1}^{N}\mathcal{C}_{g}[I_{\{ \Phi \geq \frac{iM}{N}%
\}}(X_{T})]=\mathcal{C}_{g}[\Phi_{N}(X_{T})].
\end{align*}
Thus by (\ref{EQ6}), we get $\mathcal{E}_{g}[\Phi(X_{T})]=\mathcal{C}_{g}%
[\Phi(X_{T})]$. The proof of (i) is complete.

(ii) Necessity. For the case $d=1$, since $\mathcal{H}_{2}\supset
\mathcal{H}_{1}$, we can get that $g$ is independent of $y$ and is positively
homogeneous in $z$ by (i). On the other hand,%
\[
\{l_{1}I_{\{W_{T}-W_{t}\geq a\}}+l_{2}I_{\{b\geq W_{T}-W_{t}\geq
a\}}:t<T,a<b,a,b,l_{1},l_{2}\in \mathbb{R}\} \subset \mathcal{H}_{2},
\]
by the proof of Lemma 9 in \cite{Hu1}, we can obtain $g(t,z)=g(t,1)z$. For the
case $d>1$, the proof is the same as (i).

Sufficiency. By the similar analysis as in (i), for each $\phi_{i}\in
C_{b}^{3}(\mathbb{R})$, $i=1,\ldots,N$, we can get%
\[
\mathcal{E}_{g}[\sum_{i=1}^{N}\phi_{i}(X_{T})]=\sum_{i=1}^{N}\mathcal{E}%
_{g}[\phi_{i}(X_{T})].
\]
The same analysis as in (i), we only need to prove the result for%
\[
\Phi(x)=\sum_{i=1}^{N}b_{i}I_{A_{i}}(x),
\]
where $b_{i}\geq0$, $A_{i}\in \mathcal{B}(\mathbb{R})$ and $A_{i}\supset
A_{i+1}$. Set%
\[
P_{X_{T}}(A):=P(X_{T}^{-1}(A))\text{ \ for }A\in \mathcal{B}(\mathbb{R}),
\]
then by Lusin's theorem, we can choose $\phi_{i}^{k}\in C_{b}^{3}(\mathbb{R}%
)$, $k\geq1$, such that%
\[
E[|\phi_{i}^{k}(X_{T})-I_{A_{i}}(X_{T})|^{2}]=E_{P_{X_{T}}}[|\phi_{i}%
^{k}(x)-I_{A_{i}}(x)|^{2}]\rightarrow0\text{ \ as }k\rightarrow \infty.
\]
Thus we obtain $\mathcal{E}_{g}[\Phi(X_{T})]=\mathcal{C}_{g}[\Phi(X_{T})]$ as
in (i). The proof is complete. $\Box$ \bigskip

In the following, we consider the case $n>1$. We give the following
assumptions on $\sigma$ in SDE (\ref{SDE1}).

\begin{description}
\item[(S4)] There exists a $k\leq d$ such that $\sigma_{i}(t,x)=(\tilde
{\sigma}(t,x),0,\ldots,0)$ for $i=1,\ldots,n$, where $\sigma_{i}$ is the
$i$-th row of $\sigma$ and $\tilde{\sigma}:[0,T]\times \mathbb{R}%
^{n}\rightarrow \mathbb{R}^{1\times k}$.

\item[(S5)] There exists a $k\leq d$ such that $\sigma_{i}(t,x)=(\tilde
{\sigma}(t,x),\tilde{\sigma}_{i}(t,x))$ for $i=1,\ldots,n$, where $\sigma_{i}$
is the $i$-th row of $\sigma$, $\tilde{\sigma}:[0,T]\times \mathbb{R}%
^{n}\rightarrow \mathbb{R}^{1\times k}$ and $\tilde{\sigma}_{i}:[0,T]\times
\mathbb{R}^{n}\rightarrow \mathbb{R}^{1\times(d-k)}$.
\end{description}

Set%
\[%
\begin{array}
[c]{l}%
\mathcal{H}_{3}:=\{ \xi:\exists b,\sigma \text{ satisfying (S1)-(S3), (S4) and
}x\in \mathbb{R}^{n}\text{ such that }\xi=X_{T}^{0,x}\}.\\
\mathcal{H}_{4}:=\{ \xi:\exists b,\sigma \text{ satisfying (S1)-(S3), (S5) and
}x\in \mathbb{R}^{n}\text{ such that }\xi=X_{T}^{0,x}\}.\\
\mathcal{H}_{5}:=\{ \Phi(\xi)\in L^{2}(\mathcal{F}_{T}):\Phi \text{ is
measurable on }\mathbb{R}^{n}\text{ and }\xi \in \mathcal{H}_{3}\}.\\
\mathcal{H}_{6}:=\{ \Phi(\xi)\in L^{2}(\mathcal{F}_{T}):\Phi \text{ is
measurable on }\mathbb{R}^{n}\text{ and }\xi \in \mathcal{H}_{4}\}.
\end{array}
\]
By the same analysis as in the proof of Theorem \ref{Th2} and the method in
the proof of main result in \cite{Hu1, Hu2}, we can obtain the following corollary.

\begin{corollary}
Suppose $g$ satisfies (H1)-(H3). Then

\begin{description}
\item[(i)] $\mathcal{E}_{g}[\cdot]=\mathcal{C}_{g}[\cdot]$ on $\mathcal{H}%
_{5}$ if and only if $\tilde{g}$ is independent of $y$ and is homogeneous in
$\tilde{z}$, where $\tilde{g}(t,y,\tilde{z}):=g(t,y,(\tilde{z},0,\ldots,0))$
for $(t,y,\tilde{z})\in \lbrack0,T]\times \mathbb{R}^{1+k}$;

\item[(ii)] $\mathcal{E}_{g}[\cdot]=\mathcal{C}_{g}[\cdot]$ on $\mathcal{H}%
_{6}$ if and only if $g$ is independent of $y$, $g(t,(\tilde{z},z^{\prime
}))=g_{1}(t,\tilde{z})+g_{2}(t,z^{\prime})$ for $\tilde{z}\in \mathbb{R}^{k}$,
$z^{\prime}\in \mathbb{R}^{d-k}$, $g_{1}$ is homogeneous in $\tilde{z}$ and
$g_{2}$ is linear in $z^{\prime}$.
\end{description}
\end{corollary}

\end{document}